\documentclass[10pt]{article}
\usepackage{amsmath,amsthm,amsfonts,latexsym,amsopn,verbatim,amscd,amssymb}
\newtheorem{theorem}{Theorem}[section]
\newtheorem{lemma}[theorem]{Lemma}
\newtheorem{proposition}[theorem]{Proposition}

\newtheorem{corollary}[theorem]{Corollary}

\newtheorem{remark}[theorem]{Remark}

\newtheorem{observation}[theorem]{Observation}

\numberwithin{equation}{section}

\DeclareMathOperator{\spec}{Spec}
\DeclareMathOperator{\lspec}{Lspec}
\DeclareMathOperator{\cent}{Cent}

\begin{document}
\begin{center}
		\textbf{  SPECTRUM AND ENERGY OF NON-COMMUTING GRAPHS OF FINITE GROUPS   }
     	\end{center}
     	
    
  	\begin{center}
 \textbf { Walaa Nabil Taha Fasfous$^{1}$ and Rajat Kanti Nath$^{2}$  }
  	\end{center}	
  \centerline {$^{1,2}$Department of Mathematical Sciences, Tezpur University}
   \centerline { Napaam-784028, Sonitpur, Assam, India}
  \centerline { E-mails: $^{1}$w.n.fasfous@gmail.com,  $^{2}$rajatkantinath@yahoo.com (corresponding author)}
\vspace{.3cm}

\textbf{Abstract.}
Let $G$ be a finite  non-abelian group and ${\Gamma}_{nc}(G)$ be its non-commuting graph.
In this paper, we compute spectrum and energy of ${\Gamma}_{nc}(G)$ for  certain classes of finite groups. As a consequence of our results we construct infinite families of integral complete $r$-partite graphs. 
We compare energy and Laplacian energy  (denoted by $E({\Gamma}_{nc}(G))$ and $LE({\Gamma}_{nc}(G))$  respectively) of ${\Gamma}_{nc}(G)$ and conclude that $E({\Gamma}_{nc}(G)) \leq LE({\Gamma}_{nc}(G))$ for those groups except for some non-abelian groups of order $pq$. This shows that the conjecture posed in [Gutman, I., Abreu, N. M. M., Vinagre, C. T.M., Bonifacioa, A. S and Radenkovic, S.  Relation between energy and Laplacian energy, {\em MATCH Commun.  Math.  Comput. Chem.}, {\bf 59}: 343--354, (2008)] does not hold for non-commuting graphs of certain finite groups, which also produces   new families of counter examples to the above mentioned conjecture.

\vspace{.2cm}


\section{Introduction}
Let $G$ be a finite  non-abelian group and ${\Gamma}_{nc}(G)$ be its non-commuting graph. Then ${\Gamma}_{nc}(G)$ is  a simple undirected graph whose vertex set is $G \setminus Z(G)$ (where $Z(G)$ is the center of $G$), and two vertices $a, b$ are adjacent if and only if $ab \neq ba$. The study of $\Gamma_{nc}(G)$ was originated from the works of Erd\"os  and Neumann \cite{N2}. 
Interactions between the graph theoretical properties of $\Gamma_{nc}(G)$ and the group theoretical properties of $G$ can be found in  \cite{AAM,D,MP,MSZZ}.
Properties of non-commuting graphs of some finite simple groups and dihedral groups can be found in \cite{Solomon-13, Talebi}. In recent years spectral properties of $\Gamma_{nc}(G)$ become an interesting topic of research. The complement of $\Gamma_{nc}(G)$, known as commuting graph of $G$,  is also studied widely in  past years. Various spectra and energies of commuting graphs of several classes of finite non-abelian groups were computed and analyzed in \cite{DBN,DN1,DN2,DN3,DN4,Nath-2018}. 
It is worth mentioning that there is a nice relation between the   Laplacian spectrum of a graph and the  Laplacian spectrum of its complement (see \cite[Theorem 3.6]{Mohar-91}). Exhibiting this relation, in \cite{DN-2018} and \cite{DN}, Laplacian  spectrum and  Laplacian energy  of  ${\Gamma}_{nc}(G)$ were computed for several classes of finite non-abelian groups. Unfortunately, there is no similar relation between  the spectrum of a graph and the spectrum of its complement.  This motivates us to compute  spectrum and energy of ${\Gamma}_{nc}(G)$ though the spectrum and energy of commuting graphs of several classes of finite non-abelian groups are known. 
A few results on various spectra and energies of non-commuting graphs of finite non-abelian groups can also be found in \cite{AEN,GG,GGB, Mse2017}.

 The spectrum of a graph $\mathcal{G}$ denoted by $\spec(\mathcal{G})$ is the set  $\{\lambda_{1}^{k_1}, \lambda_{2}^{ k_2}, \dots, \lambda_{n}^{k_n}\}$, where $\lambda_1, \lambda_2, \dots, \lambda_n$ are the eigenvalues of the adjacency matrix of $\mathcal{G}$ (denoted by ${\mathcal{A}}(\mathcal{G})$) with multiplicities $k_1, k_2, \dots, k_n$, respectively. Similarly, the Laplacian spectrum of  $\mathcal{G}$ is the set  $\lspec({\mathcal{G}}) = \{\mu_{1}^{l_1}, \mu_{2}^{l_2}, \dots, \mu_{m}^{ l_m}\}$, where $\mu_1, \mu_2, \dots, \mu_m$, with multiplicities $l_1, l_2, \dots, l_m$ respectively, are the eigenvalues of the Laplacian matrix of $\mathcal{G}$ given by   $L(\mathcal{G}) = D(\mathcal{G}) - {\mathcal{A}}(\mathcal{G})$. Here $D(\mathcal{G})$ stands for the degree matrix of $\mathcal{G}$.    The energy  and Laplacian energy  of a graph $\mathcal{G}$ are given by  
\[
E(\mathcal{G}) = \sum_{\lambda \in \spec(\mathcal{G})} {|\lambda|} \text{ and }
LE(\mathcal{G}) = \sum_{\mu \in \lspec({\mathcal{G}})}\left|\mu - \frac{2|e(\mathcal{G})|}{|v(\mathcal{G})|}\right|
\]	
respectively, where $v(\mathcal{G})$ and $e(\mathcal{G})$ denotes the set of vertices and edges of $\mathcal{G}$. In \cite{GAVBR}, it was conjectured that
\begin{equation}\label{conjecture 1}
E(\mathcal{G}) \leq LE(\mathcal{G}).
\end{equation}
However, \eqref{conjecture 1} was disproved in \cite{LL, SSM}. One of the objectives of this paper is to check whether \eqref{conjecture 1} holds for non-commuting graphs of finite non-abelian groups. 

In Section 2, we first compute  $E({\Gamma}_{nc}(G))$ for some  finite non-abelian groups (viz. $V_{8n}$ for odd $n$, $SD_{8n}$ and groups for which $\frac{G}{Z(G)} \cong {\mathbb{Z}}_p \times {\mathbb{Z}}_p$) whose $\spec({\Gamma}_{nc}(G))$ are known. After that we consider non-commuting graphs of the group of order $pq$, Hanaki groups, and groups  for which $\frac{G}{Z(G)} \cong D_{2m}$ and compute their spectrum and energy. It is also illustrated that using our results one may compute $\spec({\Gamma}_{nc}(G))$ and   $E({\Gamma}_{nc}(G))$ for the groups $M_{2rs}, D_{2m}, Q_{4m}$  and $U_{6n}$. In Section 3, it is  observed that there are infinite families of finite non-abelian groups such that  ${\Gamma}_{nc}(G)$ is integral. A graph is called integral if $\spec(\mathcal{G})$ contains only integers. Since it is difficult to characterize integral graphs in general, people restrict the study of integral graph for some particular families of graphs. In \cite{Roitman-84} Roitman constructed  an infinite family of integral complete tripartite graphs and remarked about the existence  of   infinite families of integral complete $r$-partite graph for $r > 3$. In \cite{Wlh-04}, Wang,  Li and Hoede, obtained a necessary and sufficient condition for integral complete $r$-partite graph. As a consequence of our results we also  construct infinite families of integral complete $r$-partite graphs.  
In Section 4, we compare $E({\Gamma}_{nc}(G))$ and $LE({\Gamma}_{nc}(G))$ for all the groups  considered in Section 2 and it is observed that $E({\Gamma}_{nc}(G)) \leq LE({\Gamma}_{nc}(G))$ if $|G| \ne pq$ for some primes $p$ and $q$. This shows that \eqref{conjecture 1} does not hold  for non-commuting graphs of  some  non-abelian groups of order $pq$, which produces new families of counter examples for \eqref{conjecture 1}.


%
We write  $\chi_M(\lambda)$ to denote the characteristic polynomial (in $\lambda$) of a square matrix $M$ of size $n$ and so $\chi_M(\lambda) = \det(\lambda I_{n \times n} - M)$, where $I_{n \times n}$ is the identity matrix of size $n$.  
Let $J_{m \times n}$ be the  matrix with all entries one. It is well-known that 
\begin{equation}\label{chi-Jn}
\chi_{J_{n \times n}}(\lambda) = \lambda^{n - 1}(\lambda - n).
\end{equation}
We also have the following useful results.  
\begin{lemma}
\label{lem1} {\rm {\cite{CRS} }}
Let $N$ be the block matrix given by
\[
N =  \begin{bmatrix}
               0_{m \times m}          & B_{m \times n} \\
               B^{T} _{m \times n}     & A_{n \times n}
              \end{bmatrix}.
\]
 Then 
\[
\chi_N(\lambda) =  \lambda^{m-n}\det(\lambda^2 I_{n\times n} - \lambda A_{n \times n} - B_{m\times n}^{T}B_{m\times n}),  
\]
where $B_{m\times n}^{T}$ is the transpose of $B_{m\times n}$. 
 \end{lemma} 
Let  $M$ = $[m_{ij}]$ and $N$ = $[n_{ij}]$ be matrices of size $a \times p$ and $q \times b$, respectively. The tensor product (or Kronecker product) of $M$ and $N$, denoted by $M \otimes N$, is the matrix of size $aq \times pb$   obtained from $M$ by replacing each entry $m_{ij}$ of $M$ with the $q \times b$ matrix $m_{ij}N$.

\begin{theorem}\label{thm1} {\rm {\cite{BC} }}
Let $X$ and $Y$ be square matrices of size $m$ and $n$, respectively. If $\lambda_1, \lambda_2, \dots, \lambda_m$ are eigenvalues of $X$ and $\mu_1, \mu_2, \dots, \mu_n$ are eigenvalues of $Y$, then  the  eigenvalues of $X \otimes Y$ are given by  
$\lambda_i\mu_j$ for $1\leq i\leq m, 1\leq j\leq n$. 
\end{theorem}

 \section{Spectra and energies}
%
We begin this section by computing energies of non-commuting graphs of  finite groups such that  spectra  of their non-commuting graphs are known (already computed in \cite {GGB}).  
\begin{theorem} \label{thm3}
Let $G$ be a finite group. 
\begin{enumerate}
\item If   $V_{8n} = \langle a, b : a^{2n} = b ^4 =1, b^{-1}ab^{-1} = bab = a^{-1}\rangle$ (where $n$ is odd) then  
\[
E(\Gamma_{nc}(V_{8n})) = 2(2n-1) + 2\sqrt{(2n - 1)(10n - 1)}.
\]
\item If $SD_{8n} = \langle a, b : a^{4n} = b ^2 =1, bab = a^{2n-1}\rangle$ then 
\[
E(\Gamma_{nc}(SD_{8n})) = \begin{cases}
2(n-1) + 4\sqrt{(n-1)(5n-1)}, &\text{ if $n$ is odd}\\
2(2n-1) + 2\sqrt{(2n - 1)(10n - 1)}, &\text{ if $n$ is even}.
\end{cases}
\]
\item If $\frac{G}{Z(G)} \cong {\mathbb{Z}}_p \times {\mathbb{Z}}_p$, where $p$ is any prime, then 
\[
E(\Gamma_{nc}(G) = 2p(p - 1)|Z(G)|.
\]
In particular, if $G$ is non-abelian and $|G| = p^3$ then $E(\Gamma_{nc}(G) = 2p^2(p - 1)$.
\end{enumerate}
\end{theorem}
\begin{proof}

(a) By \cite[Theorem 3.11]{GGB} we have
\begin{equation}\label{spectrum-V8n}
\spec(\Gamma_{nc}(V_{8n}))= \left\{(-2)^{2n-1}, 0^{6n-3}, \left((2n-1) \pm \sqrt{(2n - 1)(10n - 1)}\right)^1 \right\}.
\end{equation}
Therefore,
$E(\Gamma_{nc}(V_{8n})) =  2(2n - 1) + (2n-1) + \sqrt{(2n - 1)(10n - 1)} - (2n-1) + \sqrt{(2n - 1)(10n - 1)}$ and hence the result follows.

(b) By \cite[Theorem 3.12]{GGB} we have
\begin{equation}\label{spectrum-SD8n-n-odd}
\spec(\Gamma_{nc}(SD_{8n}))= \left\{(-4)^{2n-1}, 0^{7n-5}, \left(2(n-1) \pm 2\sqrt{(5n-1)(n-1)}\right)^1 \right\}
\end{equation}
if $n$ is odd and
\begin{equation}\label{spectrum-SD8n-n-even}
\spec(\Gamma_{nc}(SD_{8n}))= \left\{(-2)^{2n-1}, 0^{6n-3}, \left((2n-1) \pm \sqrt{(2n - 1)(10n - 1)}\right)^1\right\}
\end{equation}
if $n$ is even.
Therefore, $E(\Gamma_{nc}(SD_{8n})) =  4(2n-1) + 2(n-1) + 2\sqrt{(5n-1)(n-1)} - 2(n-1) + 2\sqrt{(5n-1)(n-1)}$ or $2(2n-1) + (2n-1) + \sqrt{(2n - 1)(10n - 1)} - (2n-1) + \sqrt{(2n - 1)(10n - 1)}$ according as  $n$ is odd or even.
Hence the result follows.

(c) By \cite[Theorem 3.4]{GGB} we have
\begin{equation}\label{spectrum-ZpxZp}
\spec(\Gamma_{nc}(G))= \left\{(-(p - 1)|Z(G)|)^{p}, 0^{(p + 1)((p - 1)|Z(G)| - 1)}, (p(p - 1)|Z(G)|)^1 \right\}.
\end{equation}
Therefore,
$E(\Gamma_{nc}(G)) =  p(p - 1)|Z(G)| + p(p - 1)|Z(G)|$ and hence the result follows.
\end{proof}
Consider the Frobenious group $F_{p, q} = \langle a,b : a^p = b^q = 1, b^{-1}ab = a^u\rangle$ of order $pq$, where $p$ and $q$ are primes with $q\mid(p - 1)$ and $u$ is an integer such that $\bar{u} \in {\mathbb{Z}}_p^{*}$ having order $q$. It is worth mentioning that the spectrum of non-commuting graph of $F_{p, q}$ was  computed in   \cite[Theorem 3.13]{GGB}. However, there is a typo in  \cite[Theorem 3.13]{GGB}.  The correct expression for   spectrum of non-commuting graph of $F_{p, q}$  is given  by
\[
\spec(\Gamma_{nc}(F_{p, q}))= \left\lbrace 0^{pq-p-2}, (-(q-1))^{p-1}, \left(\frac{\alpha \pm \sqrt{\alpha^2 +4p\alpha}}{2}\right)^1\right\rbrace,
\]
where $\alpha = (p-1)(q-1)$. Therefore,
\begin{equation}\label{thm11}
E(\Gamma_{nc}(F_{p, q}))= \alpha + \sqrt{\alpha^2 +4p\alpha}.
\end{equation}
 

Now we consider certain finite groups such that the spectra of their  non-commuting graphs  are not computed so far. The following lemma is very useful in this regard.
\begin{lemma}\label{lem7} 
Consider the matrix $J_{m \times n}$ and the block matrix 
\[
B_{(m + n) \times (m + n)} =  \begin{bmatrix}
               0_{m \times m}          & J_{m \times n} \\
               J_{n \times m}          & (J - I)_{n \times n} \\
              \end{bmatrix}.
\]
 Then

$\spec(J_{n \times n}) = \{0^{n-1}, n^1\}$ and $\spec(B_{(m + n) \times (m + n)}) = \left\lbrace 0^{m - 1}, (-1)^{n - 1}, \left(\frac{(n - 1) \pm \sqrt{(n - 1)^2 + 4mn}}{2}\right)^1\right\rbrace$.
\end{lemma}
\begin{proof}
The first part is well-known (also follows from \eqref{chi-Jn}). For the second part, by Lemma \ref{lem1}, we have
\begin{align*}
\chi_{B_{(m + n) \times (m + n)}}(\lambda)  &=  \lambda^{m-n} \det\left(\lambda^2 I_{n \times n} -\lambda(J - I)_{n \times n}- J_{n \times m} \times J_{m \times n}\right)\\
&= \lambda^{m-n} \det\left((\lambda^2 + \lambda )I_{n \times n}- (\lambda + m )J_{n \times n}\right)\\
&= \lambda^{m-n}(\lambda + m )^n \det\left(\frac{\lambda^2   + \lambda}{\lambda + m}I_{n \times n} - J_{n \times n}\right)\\
&= \lambda^{m-n}(\lambda + m )^n\chi_{J_{n \times n}}\left(\frac{\lambda^2   + \lambda}{\lambda + m}\right).
\end{align*}
Therefore, by \eqref{chi-Jn}, we have
\[
\chi_{B_{(m + n) \times (m + n)}}(\lambda) = \lambda^{m-1} (\lambda +1)^{n - 1} (\lambda^2- (n - 1)\lambda - mn).
\]
Hence, the result follows. 
%
%
%
%
%
\end{proof}

\begin{theorem}\label{thm12}
If $G$ is isomorphic to the Hanaki group 
\[
A(n,v) =  \left\{
      U(a,b) =
     \begin{bmatrix}
               1     & 0    & 0  \\
               a     & 1    & 0  \\
               b     & v(a)    & 1
       \end{bmatrix}
      : a,b \in GF(2^n) 
       \right\}, 
\]
where $n \geq 2$ and $v$ is the Frobenius automorphism of $GF(2^n)$, then 
\[
\spec(\Gamma_{nc}(G))= \left\lbrace0^{2^{2n} - 2^{n+1} +1}, (-2^n)^{2^n -2},  (2^n(2^n - 2))^1\right\rbrace
\]
 and 
\[
E(\Gamma_{nc}(A(n,v)))= 2^{n + 1}(2^{n} -2).
\] 
\end{theorem} 
\begin{proof} 
%
%
%
By  \cite[Lemma 4.2]{Fsn-19} it follows that $\Gamma_{nc}(G)$ is the complement of $(2^n - 1)K_{2^n}$ and so it is a complete $(2^n - 1)$-partite graph.     Therefore,
\begin{align*}
\mathcal{A}(\Gamma_{nc}(G)) &=    \begin{bmatrix}
      0_{1 \times 1}                           & J_{1  \times (2^n - 2)}\\
      J_{(2^n - 2) \times 1}       & (J - I)_{(2^n - 2) \times(2^n - 2)}
      \end{bmatrix}  \otimes   J_{2^n \times 2^n} \\
&= B_{(2^n - 1) \times  (2^n - 1)}  \otimes  J_{2^n \times 2^n}.
\end{align*}
Using Lemma \ref{lem7} we have  
\[
\spec(J_{2^n \times 2^n}) = \{0^{2^n -1 }, (2^n)^1\}  \text{ and } 
\]
\begin{align*}
\spec\left(B_{(2^n - 1) \times  (2^n - 1)}\right) &= \left\lbrace(-1)^{2^n -3 }, \left(\frac{(2^n -3) \pm  \sqrt{(2^n -3)^2 +4(2^n -2)}}{2}\right)^1\right\rbrace\\
& = \left\lbrace(-1)^{2^n -2}, (2^n - 2)^1\right\rbrace.  
\end{align*}
%
Therefore, by Theorem \ref{thm1}, we have 
\[
\spec(\Gamma_{nc}(G))= \left\lbrace 0^{2^{2n} - 2^{n+1} +1}, (-2^n)^{2^n -2},  (2^n(2^n - 2))^1\right\rbrace.
\]
Hence, $E(\Gamma_{nc}(G))  =   (2^n - 2)2^n + 2^n(2^n -2) = 2^{n+1}(2^n - 2).$  
%
%
%
%
%
%
%
%
\end{proof}

\begin{theorem}\label{thm13}
If $G$ is isomorphic to the Hanaki group 
\[
A(n,p) = \left\{
     V(a,b,c) =
     \begin{bmatrix}
               1     & 0    & 0  \\
               a     & 1    & 0  \\
               b     & c    & 1
       \end{bmatrix}
      : a,b,c \in GF(p^n)  
       \right\},
\]
where $p$ is a prime, then 
\[
\spec(\Gamma_{nc}(A(n,p)))= \left\lbrace0^{p^{3n} - 2p^{n} -1}, (-(p^{2n} - p^{n}))^{p^n}, \left( p^n(p^{2n} - p^{n}) \right)^1\right\rbrace
\]
and
\[
E(\Gamma_{nc}(A(n,p)))= 2p^{2n}(p^{n} - 1).
\]
\end{theorem} 
\begin{proof}
By \cite[Lemma 4.2]{Fsn-19} it follows that $\Gamma_{nc}(G)$ is the complement of $(p^n + 1)K_{p^{2n} - p^n}$ and so it is a complete $(p^n + 1)$-partite graph.     Therefore,    
\begin{align*}
\mathcal{A}(\Gamma_{nc}(G)) &= \begin{bmatrix}
      0_{1\times 1}                           & J_{1  \times p^n }\\
      J_{p^n \times 1}                & (J - I)_{p^n \times p^n}
      \end{bmatrix}  \otimes   J_{(p^{2n} - p^n) \times (p^{2n} - p^n)} \\
&= B_{(p^n + 1) \times (p^n + 1)}  \otimes  J_{(p^{2n} - p^n) \times (p^{2n} - p^n)}.
\end{align*}
Using Lemma \ref{lem7} we have 
\[
\spec\left(J_{(p^{2n} - p^n) \times (p^{2n} - p^n)}\right) =  \{0^{p^{2n} - p^n -1 }, ((p^{2n} - p^n))^1\} \text{ and}
\] 
\begin{align*}
\spec\left(B_{(p^n + 1) \times (p^n + 1)}\right) &= \left\lbrace (-1)^{p^n -1}, \left(\frac{(p^n -1) \pm \sqrt{(p^n -1)^2 +4p^n }}{2}\right)^1\right\rbrace\\
&= \left\lbrace (-1)^{p^n}, (p^n)^1 \right\rbrace.
\end{align*}
Therefore, by Theorem \ref{thm1},  we have
\[
\spec(\Gamma_{nc}(A(n,p)))= \left\lbrace 0^{p^{3n} - 2p^{n} -1}, (-(p^{2n} - p^{n}))^{p^n}, \left( p^n(p^{2n} - p^{n}) \right)^1\right\rbrace.
\]
Hence, $E(\Gamma_{nc}(A(n,p)))= p^n(p^{2n} - p^{n}) + p^n(p^{2n} - p^{n})  = 2p^{2n}(p^{n} - 1).$
\end{proof}

\begin{theorem}\label{thm10}
Let $G$ be a finite group with center $Z(G)$. If $\frac{G}{Z(G)} \cong D_{2m}$, where $D_{2m} = \langle a, b: a^m = b^2 =1, bab^{-1}=a^{-1}\rangle$, then 
\[
\spec(\Gamma_{nc}(G))= \left\lbrace 0^{2nm-n-m-1}, (-n)^{m-1}, \left(\frac{n(m-1) \pm n\sqrt{(m-1)(5m-1)}}{2}\right)^1\right\rbrace
\]
and 
\[
E(\Gamma_{nc}(G))= n\left((m-1) + \sqrt{(m-1)(5m-1)}\right),
\]
where $n = |Z(G)|$.
\end{theorem} 
\begin{proof}
Let $\frac{G}{Z(G)} = \langle xZ(G), yZ(G) : x^2Z(G) = y^mZ(G) = Z(G), xyx^{-1}Z(G) = y^{-1}Z(G)\rangle$.
It can be seen that the  independent sets of $\Gamma_{nc}(G)$  are given by
$V_0 := yZ(G)\sqcup \cdots \sqcup y^{m-1}Z(G)$, 
$V_1 := xy^1Z(G)$,
$V_2 := xy^2Z(G)$, $\dots$, 
$V_m := xy^mZ(G)$.
By \cite[Lemma 3.1]{Fsn-19}, it follows that $\Gamma_{nc}(G)$ is the complement of $K_{(m-1)|Z(G)|} \sqcup mK_{|Z(G)|}$ and so it is a complete $(m + 1)$-partite graph. Hence,    
\begin{align*}
\mathcal{A}(\Gamma_{nc}(G)) &= \begin{bmatrix}
      0_{(mn-n) \times (mn-n)}  & J_{(mn-n) \times n}  & J_{(mn-n) \times n}  & J_{(mn-n) \times n}  & J_{(mn-n) \times n}  & J_{(mn-n) \times n} \\
      J_{n \times (mn-n)}  & 0_{n \times n}  & J_{n \times n}  & J_{n \times n}  & J_{n \times n}  & J_{n \times n}\\
      J_{n \times (mn-n)}  & J_{n \times n}  & 0_{n \times n}  & J_{n \times n}  & J_{n \times n}  & J_{n \times n}\\
       \vdots\\
      J_{n \times (mn-n)}  & J_{n \times n}  & J_{n \times n}  & J_{n \times n}  & J_{n \times n}  & 0_{n \times n}\\
        \end{bmatrix}\\
&= \begin{bmatrix}
      0_{(m-1) \times (m-1)}                & J_{(m-1) \times m}\\
      J_{m \times (m-1)}       & (J - I)_{m \times m}
      \end{bmatrix} \otimes   J_{n \times n}\\
&= B_{(2m - 1) \times (2m - 1)}\otimes   J_{n \times n},      
\end{align*}
where $n = |Z(G)|$. By Lemma \ref{lem7} we have
\begin{align*}
\spec(J_{|Z(G)| \times |Z(G)|}) &= \{0^{|Z(G)|-1}, (|Z(G)|)^1\}  \text{ and} \\\spec(B_{(2m - 1) \times (2m - 1)}) &= \left\lbrace0^{m-2}, (-1)^{m-1}, \left({(m-1)} \pm  {\frac{\sqrt{(m-1)(5m-1)}}{2}}\right)^1\right\rbrace.
\end{align*}  
Therefore, by using Theorem \ref{thm1}, we have    
\[
\spec(\Gamma_{nc}(G))= \left\lbrace 0^{2nm-n-m-1}, (-n)^{m-1}, \left(\frac{n(m-1) \pm n\sqrt{(m-1)(5m-1)}}{2}\right)^1\right\rbrace.
\] 
Hence, 
\begin{align*}
E(\Gamma_{nc}(G)) & = (m-1)n + \frac{n(m-1) +  n\sqrt{(m-1)(5m-1)}}{2}\\
& ~~~~~~~~~~~~~~~~ + \frac{-n(m-1) +  n\sqrt{(m-1)^2 +4m(m-1)}}{2} \\
& = n\left((m-1) + \sqrt{(m-1)(5m-1)}\right). 
\end{align*}
\end{proof}
We have the following corollary to Theorem \ref{thm10}.
\begin{corollary}\label{cor-M2rs}
Consider the group $M_{2rs} = \langle a, b : a^r = b^{2s} = 1, bab^{-1} = a^{-1} \rangle$.  
\begin{enumerate}
\item If $r$ is odd then
\[
\spec(\Gamma_{nc}(M_{2rs}))= \left\lbrace 0^{2sr-s-r-1}, (-s)^{r-1}, \left(\frac{s(r-1) \pm s\sqrt{(r-1)(5r-1)}}{2}\right)^1\right\rbrace
\]
and 
\[
E(\Gamma_{nc}(M_{2rs}))= s\left((r-1) + \sqrt{(r-1)(5r-1)}\right).
\]
\item If $r$ is even then
\[
\spec(\Gamma_{nc}(M_{2rs}))= \left\lbrace 0^{2sr-2s-\frac{r}{2}-1}, (-2s)^{\frac{r}{2}-1}, \left(s\left(\frac{r}{2} - 1\right) \pm s\sqrt{\left(\frac{r}{2} - 1\right)\left(\frac{5r}{2} - 1\right)}\right)^1\right\rbrace
\]
and 
\[
E(\Gamma_{nc}(M_{2rs}))= s\left((r-2) + \sqrt{(r-2)(5r-2)}\right).
\]
\end{enumerate}
\end{corollary}
\begin{proof}
The result follows from Theorem \ref{thm10} by putting $m=r, n = s$ if $r$ is odd and $m=\frac{r}{2}, n = 2s$ if $r$ is even.
\end{proof}

Putting $r =m$ and $s = 1$ in Corollary \ref{cor-M2rs} we get  
\begin{corollary} 
\begin{enumerate}
\item If $m$ is odd then
\[
\spec(\Gamma_{nc}(D_{2m}))= \left\lbrace 0^{m-2}, (-1)^{m-1}, \left(\frac{(m-1) \pm \sqrt{(m-1)(5m-1)}}{2}\right)^1\right\rbrace
\]
and 
\[
E(\Gamma_{nc}(D_{2m}))= \left((m-1) + \sqrt{(m-1)(5m-1)}\right).
\]
\item If $m$ is even then
\[
\spec(\Gamma_{nc}(D_{2m}))= \left\lbrace 0^{\frac{3m}{2}-3}, (-2)^{\frac{m}{2}-1}, \left(\left(\frac{m}{2} - 1\right) \pm \sqrt{\left(\frac{m}{2} - 1\right)\left(\frac{5m}{2} - 1\right)}\right)^1\right\rbrace
\]
and 
\[
E(\Gamma_{nc}(D_{2m}))= \left((m-2) + \sqrt{(m-2)(5m-2)}\right).
\]
\end{enumerate}
\end{corollary}
It is worth mentioning that the energies of  non-commuting graphs of dihedral groups ($D_{2m}$) were also obtained in \cite{Mse2017}  by direct calculation.

Putting $n = 2$ in Theorem \ref{thm10} we  get the following corollary.
\begin{corollary}
If $Q_{4m} = \langle a, b : b^{2m} = 1, a^2 = b^m, aba^{-1} = b^{-1} \rangle$, where $m \geq 2$, then
\[
\spec(\Gamma_{nc}(Q_{4m}))= \left\lbrace 0^{3m-3}, (-2)^{m-1}, \left((m-1) \pm \sqrt{(m-1)(5m-1)}\right)^1\right\rbrace
\]
and 
\[
E(\Gamma_{nc}(Q_{4m}))= 2\left((m-1) + \sqrt{(m-1)(5m-1)}\right),
\]
\end{corollary}

Putting $m = 3$ in Theorem \ref{thm10} we also get the following corollary.
\begin{corollary}
If $U_{6n} = \langle a, b : a^{2n} =1, b^3 = 1, a^{-1}ba = b^{-1}\rangle $ then  
\[
\spec(\Gamma_{nc}(U_{6n})) = \left\{0^{5n-4}, (-n)^2, (n+n\sqrt{7})^1, (n- n\sqrt{7})^1\right\}. 
\]
and
$E(\Gamma_{nc}(U_{6n})) = 2n(1+ \sqrt{7})$.
\end{corollary}
%
%
%
%
We conclude this section by noting that the spectrum and energy of  non-commuting graph of $U_{6n}$ are also obtained in \cite[Theorem 3.8]{GGB} by direct calculation.

\section{Perfect squares and integral non-commuting graphs}
By Theorem \ref{thm12} and Theorem \ref{thm13} it follows that $\Gamma_{nc}(A(n,v))$ and $\Gamma_{nc}(A(n,p))$ are integral. Also, if $\frac{G}{Z(G)} \cong {\mathbb{Z}}_p \times {\mathbb{Z}}_p$, where $p$ is any prime, then by \cite[Theorem 3.4]{GGB} we have  $\Gamma_{nc}(G)$ is integral. The following three observations on perfect squares are useful in finding more finite non-abelian groups such that their  non-commuting graphs are integral.

\begin{observation}\label{Oservation-1}
 The  positive integers $n_i$ such then $(n_i - 1)(5n_i - 1)$ is a perfect square are given by
\[
n_{i+2} = 322 n_{i+1} - n_i - 192
\]
where
\[
\text{(i)}~~  n_1 = 1, n_2 = 65 \quad \quad \text{(ii)}~~ n_1 = 2, n_2 = 442 \quad \quad \text{(iii)}~~ n_1 = 10, n_2 = 3026.
\]  
\end{observation}
\noindent The first few numbers in each of the cases are listed below:

\noindent Case $(i)$: $1, 65, 20737, 6677057, 2149991425, 692290561601, 222915410843905,$ 

~~~~~~~~$71778070001175617, \dots .$ 

\noindent Case $(ii)$: $2,  442,  142130,  45765226,  14736260450,  4745030099482,  1527884955772562,$   

~~~~~~~~~$491974210728665290, \dots .$

\noindent Case $(iii)$: $10,  3026,  974170,  313679522,  101003831722,  32522920134770,  10472279279564026,\dots .$

The three lists given by Observation \ref{Oservation-1} may be combined as given in the next observation.
\begin{observation} \label{Oservation-2}
The  positive integers $n_i$ such then $(n_i - 1)(5n_i - 1)$ is a perfect square are given by
\[
n_{i+6} = 322 n_{i+3} - n_i - 192,
\]
where $n_1= 1, \; \; n_2 = 2,  \; \; n_3 =  10,   \; \; n_4 = 65,   \; \; n_5 = 442$ and  $n_6 =3026$.
\end{observation}
The first $28$ positive numbers $n$ such that $(n - 1)(5n - 1)$ is a perfect square are given   below:
\begin{center}
\begin{tabular}{|c|c|c|c|}
\hline 
    &                &                &\\
$n$ & $\sqrt{n - 1}$ & $\sqrt{5n - 1}$& $\frac{n}{2}$ (if $n$ is even) \\ 
\hline 
$1$ & $0$ & $2$ & \\ 
\hline 
$2$ & $1$ & $3$ & $1$ \\ 
\hline 
$10$ & $3$ & $7$ & $5$ \\ 
\hline 
$65$ & $8$ & $18$ & \\ 
\hline 
$442$ &   $21$ &   $47$ & $221$\\
\hline
$3026$ &   $55$ &   $123$ &$1513$ \\
\hline
$20737$ &   $144$ &  $322$ &\\ 
\hline 
$142130$ & $377$ &  $843$ & $71065$\\
\hline 
$974170$ & $987$ & $2207$ & $487085$ \\ 
\hline
$6677057$ &   $2584$ &   $5778$ &\\ 
\hline 
$45765226$ &   $6765$ &  $15127$ & $22882613$ \\
\hline
$313679522$ &  $17711$ &   $39603$ & $156839761$ \\ 
\hline
$2149991425$ &   $46368$ &   $103682$ &  \\
\hline 
$14736260450$ & $121393$   &   $271443$ & $7368130225$ \\
\hline 
$101003831722$ &   $317811$ &   $710647$ & $50501915861$\\
\hline
$692290561601$ &   $832040$ &   $1860498$ & \\ 
\hline 
$4745030099482$ & $2178309$   & $4870847$ & $2372515049741$  \\ 
\hline 
$32522920134770$ & $5702887$  & $12752043$ & $16261460067385$ \\
\hline 
$222915410843905$ & $14930352$  & $33385282$ & \\ 
\hline 
$1527884955772562$ & $39088169$  & $87403803$ & $763942477886281$\\
\hline
$10472279279564026$ & $102334155$ & $228826127$ & $5236139639782013$\\  
\hline
$71778070001175617$ & $267914296$ & $599074578$ & \\  
\hline
$491974210728665290$ & $701408733$ & $1568397607$ & $245987105364332645$ \\  
\hline
$3372041405099481410$ & $1836311903$  & $4106118243$ & $1686020702549740705$\\ 
\hline 
$23112315624967704577$ & $4807526976$   & $10749957122$ & \\
\hline 
$158414167969674450626$ & $12586269025$  & $28143753123$ & $79207083984837225313$ \\
\hline 
$1085786860162753449802$ & $32951280099$ & $73681302247$ & $542893430081376724901$\\ 
\hline 
$7442093853169599697985$ & $86267571272$ & $192900153618$ &  \\ 
\hline 
\end{tabular} 
\end{center}

It is interesting to  observe that the numbers appeared in the second  and third columns of the above table are the every other Fibonacci number and Lucas number respectively. 


\begin{observation} \label{Oservation-3}
If $n$ is an even positive integer such that $(n - 1)(5n - 1)$ is a perfect square then $(2m - 1)(10m - 1)$ is also a perfect square for $m = \frac{n}{2}$.
\end{observation}




By Observation \ref{Oservation-1} and  Observation \ref{Oservation-3} it follows that there are infinitely many positive integers $n$ such that $(2n - 1)(10n - 1)$ is  a perfect square. Some of these numbers are given in the above table. The results obtained in Section 2 and the above observations give  many families of finite non-abelian groups whose non-commuting graphs are integral.  
\begin{theorem}\label{integral-NCG}
Let $n_i$ be  positive integers such that $n_1= 1, \; \; n_2 = 2,  \; \; n_3 =  10,   \; \; n_4 = 65,   \; \; n_5 = 442$, $n_6 =3026$ and
\[
n_{i+6} = 322 n_{i+3} - n_i - 192 \text{ for } i \geq 1.
\] 
Then 
\begin{enumerate}
\item $\Gamma_{nc}(V_{8n_i})$ (for odd $n$) is integral.
\item $\Gamma_{nc}(SD_{8n_i})$ is integral if $n_i$ is odd.
\item $\Gamma_{nc}(SD_{8m_i})$ is integral if $n_i$ is even and $m_i = \frac{n_i}{2}$.
\item $\Gamma_{nc}(G)$ is integral if $\frac{G}{Z(G)} \cong D_{2n_i}$.
\item $\Gamma_{nc}(M_{2n_is})$ is integral if $n_i$ is odd.
\item $\Gamma_{nc}(M_{2m_is})$ is integral if $m_i = 2n_i$.
\item $\Gamma_{nc}(D_{2n_i})$ is integral if $n_i$ is odd.
\item $\Gamma_{nc}(D_{2m_i})$ is integral if $m_i = 2n_i$.
\item $\Gamma_{nc}(Q_{4n_i})$ is integral.
\end{enumerate}
\end{theorem}

Recall that a complete  $r$-partite graph, denoted by $K_{n_1, n_2, \dots, n_r}$, is a graph whose vertex set is the union of $r$ disjoint subsets $V_1, V_2, \dots, V_r$ of it such that $|V_i| = n_i$ for $i = 1, 2, \dots, r$ and any two vertices of $K_{n_1, n_2, \dots, n_r}$ are adjacent if and only if they belong to different $V_i$'s. As mentioned in the Introduction it is interesting to construct infinite families of integral complete  $r$-partite graphs. Note the the non-commuting graphs of the groups considered in this paper are all complete $r$-partite for some $r$. More precisely, we have  
\begin{equation}\label{K-a,b}
\Gamma_{nc}(G) = \begin{cases}
K_{(2n).2, 1.(4n-2)}& \text{ if } G \cong V_{8n} \text{ (for odd $n$)}\\
K_{n.4, 1.(4n-4)}& \text{ if  $G \cong SD_{8n}$ and $n$ is odd}\\
K_{(2n).2, 1.(4n-2)}& \text{ if  $G \cong SD_{8n}$ and $n$ is even} \\
K_{m.|Z(G)|, 1.((m-1)|Z(G)|)}& \text{ if  $\frac{G}{Z(G)} \cong D_{2m}$} \\
K_{r.s, 1.(s(r-1))}& \text{ if   $G \cong M_{2rs}$ and $r$ is odd }\\
K_{\frac{r}{2}.(2s), 1.(2s(\frac{r}{2}-1))}& \text{ if   $G \cong M_{2rs}$ and $r$ is even }\\
K_{n.1, 1.(n-1)}& \text{ if  $G \cong D_{2n}$ and $n$ is odd}\\
K_{\frac{n}{2}.2, 1.(2(\frac{n}{2}-1))}& \text{ if  $G \cong D_{2n}$ and $n$ is even}\\
K_{n.2, 1.(2(n-1))}& \text{ if } G \cong Q_{4n},  
\end{cases}
\end{equation}
where  $K_{a_1.n_1, a_2.n_2} = K_{\underset{a_1-\text{times}}{\underbrace{n_1, \dots, n_1}}, \underset{a_2-\text{times}}{\underbrace{n_2, \dots, n_2}}}$.

In view of Theorem \ref{integral-NCG} and \eqref{K-a,b} we have the following result that gives many integral complete $r$-partite graphs.
\begin{theorem}\label{integral-K-a-b}
Let $n_i$ be  positive integers such that $n_1= 1, \; \; n_2 = 2,  \; \; n_3 =  10,   \; \; n_4 = 65,   \; \; n_5 = 442$, $n_6 =3026$ and
\[
n_{i+6} = 322 n_{i+3} - n_i - 192 \text{ for } i \geq 1.
\] 
Then 
\begin{enumerate}
\item $K_{(2n_i).2, 1.(4n_i-2)}$ is integral.
\item $K_{n_i.4, 1.(4n_i-4)}$ is integral if $n_i$ is odd.
\item $K_{(2m_i).2, 1.(4m_i-2)}$ is integral if $n_i$ is even and $m_i = \frac{n_i}{2}$.
\item $K_{n_i.|Z(G)|, 1.((n_i-1)|Z(G)|)}$ is integral if $\frac{G}{Z(G)} \cong D_{2n_i}$.
\item $K_{n_i.s, 1.(s(n_i-1))}$ is integral if $n_i$ is odd.
\item $K_{\frac{m_i}{2}.(2s), 1.(2s(\frac{m_i}{2}-1))}$ is integral if $m_i = 2n_i$.
\item $K_{n_i.1, 1.(n_i-1)}$ is integral if $n_i$ is odd.
\item $K_{\frac{m_i}{2}.2, 1.(2(\frac{m_i}{2}-1))}$ is integral if $m_i = 2n_i$.
\item $K_{n_i.2, 1.(2(n_i-1))}$ is integral.
\end{enumerate}
\end{theorem}

\section{Comparing energy and Laplacian energy}

In  2018, Dutta and Nath \cite{DN} have computed the Laplacian energy of non-commuting graphs of  several classes of finite non-abelian groups. In  this section, we compare energy  and Laplacian energy of non-commuting graphs of certain classes of  finite non-abelian  groups. We begin this section with the following computations of the  Laplacian energies of non-commuting graphs of the groups  $V_{8n}$ (for odd $n$) and $SD_{8n}$.
\begin{theorem}\label{thm20-21}
If $G = V_{8n}$ (where $n$ is odd) then $\lspec(\Gamma_{nc}(G)) = \{0, (8n-4)^{2n}, (4n)^{4n-3},  (8n-2)^{2n}\}$ and $LE(\Gamma_{nc}(G)) = \frac{8n(8n^2-8n+3)}{4n-1}$.
\end{theorem}
\begin{proof}
We have $Z(G) = \langle b^2 \rangle$. Therefore,   $|Z(G)| = 2$ and $|v(\Gamma_{nc}(V_{8n}))| = 8n - 2$.  The distinct centralizers of  $V_{8n}$ are given by
\begin{align*}
X_0 & = G\\
X_1 & = Z(G) \sqcup aZ(G) \sqcup a^2Z(G) \sqcup \cdots \sqcup a^{2n - 1}Z(G) \\
X_2 & = Z(G) \sqcup bZ(G)\\
X_3 & = Z(G) \sqcup abZ(G)\\
& \vdots\\
X_{2 n + 1} & = Z(G) \sqcup a^{2n - 1}bZ(G).
\end{align*}
It is easy to see that $V_{8n}$ is an AC-group with   $|X_1| = 4n$ and $|X_2| =  \cdots |X_{2n+1}| = 4$. Therefore,   using    \cite[Theorem 2.4]{DN-2018}, we get
\[
\lspec(\Gamma_{nc}(G)) = \{0, (8n-4)^{2n}, (4n)^{4n-3},  (8n-2)^{2n}\}.
\]
It was shown in the proof of \cite[Theorem 3.11]{GGB} that $\Gamma_{nc}(G) = K_{1.(4n - 2), (2n).2}$.   Therefore, $2|e(\Gamma_{nc}(G))|= 48n^2 -24n$. We have

$\left| 0 - \frac{2|e(\Gamma_{nc}(G))|}{|v(\Gamma_{nc}(G))|} \right| = \frac{24n^2 -12n}{4n-1}$, \quad 
${\left| (8n-4) - \frac{2|e(\Gamma_{nc}(G))|}{|v(\Gamma_{nc}(G))|} \right|} = \frac{8n^2 -12n+4}{4n-1}$, 
\quad 
${\left| 4n - \frac{2|e(\Gamma_{nc}(G))|}{|v(\Gamma_{nc}(G))|} \right|} = \frac{8n^2 -8n}{4n-1}$, 
and  
${\left| (8n-2) - \frac{2|e(\Gamma_{nc}(G))|}{|v(\Gamma_{nc}(G))|} \right|}$ =$\frac{8n^2 -4n+2}{4n-1}$. 
Therefore, 
\[
LE(\Gamma_{nc}(G)) = \frac{24n^2 -12n}{4n-1} +  \frac{2n(8n^2 -12n+4)}{4n-1} +  \frac{(4n-3)(8n^2 -8n)}{4n-1} +  \frac{2n(8n^2 -4n+2)}{4n-1} 
\]    
and the result follows on simplification.
\end{proof}
\begin{theorem}\label{thm23}
If $G = SD_{8n}$, then
\begin{enumerate}
\item $\lspec(\Gamma_{nc}(G)) = \{0, (8n-8)^{3n}, (4n)^{4n-5},  (8n-4)^n\}$ and  $LE(\Gamma_{nc}(G)) = \frac{8n(4n^2-10n+7)}{2n-1}$ if   $n$ is odd. 
\item     $\lspec(\Gamma_{nc}(G))= \{0,(8n-4)^{2n}, (4n)^{4n-3},  (8n-2)^{2n}\}$ and  $LE(\Gamma_{nc}(G))=\frac{8n(8n^2-8n+3)}{4n-1}$ if $n$ is even.
\end{enumerate}       
\end{theorem}
\begin{proof}
We have  
\[
Z(G) = \begin{cases}
\langle a^{2n}\rangle, &\text{ if $n$ is even}\\
\langle a^{n}\rangle, &\text{ if $n$ is odd}
\end{cases} \text{ and so } |v(\Gamma_{nc}(G))| = \begin{cases}
8n - 2, &\text{ if $n$ is even}\\
8n - 4, &\text{ if $n$ is odd},
\end{cases}
\]
since $|v(\Gamma_{nc}(G))|$ $= 8n$ and  $|Z(G)| = \begin{cases}
2, &\text{ if $n$ is even}\\
4, &\text{ if $n$ is odd}.
\end{cases}$ 

(a) We see that 
\begin{align*}
X_0 &=  G\\
X_1 &=  Z(G) \sqcup aZ(G) \sqcup a^2Z(G) \sqcup \cdots \sqcup a^{n - 1}Z(G) \\ 
X_2 &=  Z(G) \sqcup bZ(G)\\
X_3 &=  Z(G) \sqcup abZ(G)\\
& \vdots\\
X_{n + 1}  &=  Z(G) \sqcup a^{n - 1}bZ(G)
\end{align*}
give all the centralizers of $G$ and they are abelian. Therefore,   using    \cite[Theorem 2.4]{DN-2018}, we get
\[
\lspec(\Gamma_{nc}(G)) = \{0, (8n-8)^{3n}, (4n)^{4n-5},  (8n-4)^n\},
\]
since   $|X_1| = 4n$ and $|X_2|=   \cdots = |X_{2n+1}|=8$.

It was shown in the proof of \cite[Theorem 3.12]{GGB} that $\Gamma_{nc}(G) = K_{1.(4n - 4), n.4}$.   Therefore, $2|e(\Gamma_{nc}(G))|= 48n^2 - 48n$. We have

 $\left| 0 - \frac{2|e(\Gamma_{nc}(G))|}{|v(\Gamma_{nc}(G))|} \right|$ =$\frac{12n^2 -12n}{2n-1}$, \quad
${\left| (8n-8) - \frac{2|e(\Gamma_{nc}(G))|}{|v(\Gamma_{nc}(G))|} \right|}$ =$\frac{4n^2 -12n+48}{2n-1}$, \quad 
 ${\left| 4n - \frac{2|e(\Gamma_{nc}(G))|}{|v(\Gamma_{nc}(G))|} \right|}$ = $\frac{4n^2 -8n}{2n-1}$,  and
${\left| (8n-4) - \frac{2|e(\Gamma_{nc}(G))|}{|v(\Gamma_{nc}(G))|} \right|}$ =$\frac{4n^2 -4n+4}{4n-1}$. Therefore,
\[  
LE(\Gamma_{nc}(G)) = \frac{12n^2 -12n}{2n-1} +  \frac{3n(4n^2 -12n+8)}{2n-1}+    \frac{(4n-5)(4n^2 -8n)}{2n-1} +   \frac{n(4n^2 -4n+4)}{2n-1}
\]
and hence the result follows on simplification.

(b) We see that 
\begin{align*}
X_0 &=  G\\
X_1 &=  Z(G) \sqcup aZ(G) \sqcup a^2Z(G) \sqcup \cdots \sqcup a^{2n - 1}Z(G) \\ 
X_2 &=  Z(G) \sqcup bZ(G)\\
X_3 &=  Z(G) \sqcup abZ(G)\\
& \vdots\\
X_{2n + 1}  &=  Z(G) \sqcup a^{2n - 1}bZ(G)
\end{align*}
give all the centralizers of $G$ and they are abelian. Therefore,   using    \cite[Theorem 2.4]{DN-2018}, we get
\[
\lspec(\Gamma_{nc}(G))= \{0,(8n-4)^{2n}, (4n)^{4n-3},  (8n-2)^{2n}\},
\]
since   $|X_1| = 4n$ and $|X_2|=   \cdots = |X_{2n+1}|=4$. 

It was shown in the proof of \cite[Theorem 3.12]{GGB} that $\Gamma_{nc}(G) = K_{1.(4n - 2), (2n).2}$.   Therefore, $2|e(\Gamma_{nc}(G))|= 48n^2 -24n$. We have

$\left| 0 - \frac{2|e(\Gamma_{nc}(G))|}{|v(\Gamma_{nc}(G))|} \right| = \frac{24n^2 -12n}{4n-1}$, \quad 
${\left| (8n-4) - \frac{2|e(\Gamma_{nc}(G))|}{|v(\Gamma_{nc}(G))|} \right|} = \frac{8n^2 -12n+4}{4n-1}$, 
\quad 
${\left| 4n - \frac{2|e(\Gamma_{nc}(G))|}{|v(\Gamma_{nc}(G))|} \right|} = \frac{8n^2 -8n}{4n-1}$, 
and  
${\left| (8n-2) - \frac{2|e(\Gamma_{nc}(G))|}{|v(\Gamma_{nc}(G))|} \right|}$ =$\frac{8n^2 -4n+2}{4n-1}$. 
Therefore, 
\[
LE(\Gamma_{nc}(G)) = \frac{24n^2 -12n}{4n-1} +  \frac{2n(8n^2 -12n+4)}{4n-1} +  \frac{(4n-3)(8n^2 -8n)}{4n-1} +  \frac{2n(8n^2 -4n+2)}{4n-1} 
\]    
and the result follows on simplification.
\end{proof}

\begin{theorem}
If $G = V_{8n}$ (for odd $n$) or $SD_{8n}$ then $E(\Gamma_{nc}(G)) \leq LE(\Gamma_{nc}(G))$.
\end{theorem}
\begin{proof}  
For any positive integer $n$ we have $64n^4 - 80 n^3 + 152n^2 -82n +4 \geq 0$ and so
\begin{align*}
20n^2 -12n +1 &\leq 64n^4 - 80 n^3 + 172n^2 -94n +5 \\
              & = 4n^2(16n^2 - 20 n) + 172n^2 -94n +5\\ 
              & = 4n^2(4n-5)^2 + (4n-5)(18n-1). 
\end{align*}
Since $16n^2$ and $(18n-1)^2$ are positive integers, we have   
\[
16n^2(20n^2 -12n +1) \leq 16n^2(4n^2(4n-5)^2 + (4n-5)(18n-1)) +(18n-1)^2.
\]
Therefore,
\begin{align*}
16n^2(20n^2 -12n +1) -(8n-1)(20n^2 &-12n +1) \\
& \leq 16n^2(4n^2(4n-5)^2 + (4n-5)(18n-1)) +(18n-1)^2 \\
\Rightarrow(4n-1)^2(20n^2 -12n +1) &\leq (8n^2(4n-5) + (18n-1)))^2\\
\Rightarrow20n^2 -12n +1 &\leq \frac{(8n^2(4n-5) + (18n-1)))^2}{(4n-1)^2}\\
\Rightarrow\sqrt{20n^2 -12n +1} &\leq \frac{8n^2(4n-5) + (18n-1))}{4n-1}\\
\Rightarrow2\sqrt{20n^2 -12n +1} &\leq \frac{64n^3- 64n^2 +24n -16n^2 +12n -2}{(4n-1)}\\ 
                      & = \frac{8n(8n^2-8n+3) -(4n-1)(4n-2)}{(4n-1)}. 
\end{align*}
Thus
\begin{equation}\label{E-LE-comp-V8n}
2\sqrt{20n^2 -12n +1} + 2(2n - 1) \leq \frac{8n(8n^2-8n+3)}{(4n-1)}.
\end{equation}
Hence, by Theorem \ref{thm3} and    Theorem \ref{thm20-21}, we have
\[
E(\Gamma_{nc}(V_{8n})) \leq LE(\Gamma_{nc}(V_{8n}))
\]
if $n$ is odd.

Suppose that $G = SD_{8n}$. If $n$ is even then by \eqref{E-LE-comp-V8n}, Theorem \ref{thm3}(b) and Theorem \ref{thm23}(b) we have
\[
E(\Gamma_{nc}(SD_{8n})) \leq LE(\Gamma_{nc}(SD_{8n})).
\]
If $n$ is odd then we have $64n^6 - 336n^5 + 668n^4 - 612n^3 + 225n^2 - 4n - 1 > 0$. Therefore,
\begin{align*}
20n^4 - 44n^3 + & 33n^2 - 10n + 1 \\
& < 64n^6 - 336n^5 + 688n^4 - 656n^3 + 258n^2 - 14n\\
& = 4n^2(16n^4 - 80n^3 + 162n^2 - 142n + 49) +
2n(20n^3 - 44n^2 + 31n - 7 - 8n^4)\\
& = 4n^2((4n^2 - 10n)^2+14(4n^2 - 10n) + 49 + 6n^2 - 2n) +
2n(20n^3 - 44n^2 + 31n - 7 - 8n^4)\\
& = 4n^2((4n^2 - 10n + 7)^2 + 2n(3n - 1)) + 2n((10n - 7)(2n^2 - 3n + 1) - 8n^4)\\
& = 4n^2(4n^2 - 10n + 7)^2 + 8n^3(3n - 1) - 16n^5 + 2n(10n - 7)(2n^2 - 3n + 1)\\
& = 4n^2(4n^2 - 10n + 7)^2 - 8n^3(2n^2 - 3n + 1) + 2n(10n - 7)(2n^2 - 3n + 1)\\
& = 4n^2(4n^2 - 10n + 7)^2 - 2n(4n^2 - 10n + 7)(2n^2 - 3n + 1)\\
& = 4n^2(4n^2 - 10n + 7)^2 - 2n(4n^2 - 10n + 7)(2n - 1)(n - 1).
\end{align*}
Since $20n^4 - 44n^3 +  33n^2 - 10n + 1 = (5n - 1)(n - 1)(2n - 1)^2$ we have
\[
(5n - 1)(n - 1)(2n - 1)^2 < 4n^2(4n^2 - 10n + 7)^2 - 2n(4n^2 - 10n + 7)(2n - 1)(n - 1).
\] 
Multiplying both sides by $4$ and then adding $(2n - 1)^2(n - 1)^2$ in the right side of the above inequality we get
\begin{align*}
4(5n - 1)(n - 1)(2n - 1)^2 < \left(4n(4n^2 - 10n + 7) - (2n - 1)(n - 1)\right)^2. 
\end{align*}
Therefore,
\[
2\sqrt{(5n - 1)(n - 1)}(2n - 1) < 4n(4n^2 - 10n + 7) - (2n - 1)(n - 1)
\]
which gives 
\[
(n - 1) + 2\sqrt{(5n - 1)(n - 1)} < \frac{4n(4n^2 - 10n + 7)}{(2n - 1)}.
\]
Hence, the result follows from Theorem \ref{thm3}(b) and Theorem \ref{thm23}(a).
\end{proof}

\begin{theorem}\label{compare-QD}
If $QD_{2^n} := \langle a, b : a^{2n-1} = b ^2 =1,  bab^{-1} = a^{2^{n-2}-1}\rangle$, where $n \geq 4$, then $E(\Gamma_{nc}(QD_{2^n})) \leq LE(\Gamma_{nc}(QD_{2^n}))$.
\end{theorem}
\begin{proof}
For all $n \geq 4$ we have $(2^{2n-3} - 2^n+3)^2 - 2^{3n-4} - 1 > 0$, $2^{2n-2} + 2^{2n-3} - 5.2^{2n-4} > 0$ and  $6.2^{n-2} - 2^{n-1} > 0$. Therefore, 

%
%
%
%
%
%

\[
(2^{2n-3} - 2^n+3)^2 + 2^{2n-2} + 2^{2n-3}-2^{n-1} -2^{3n-4}- 5.2^{2n-4}+ 6.2^{n-2} -1 > 0
\]
and so 
\[
5.2^{2n-4}- 6.2^{n-2} +1 < (2^{2n-3} - 2^n+3)^2 + 2^{2n-2} + 2^{2n-3}-2^{n-1} -2^{3n-4}.
\]
Multiplying both sides by $2^{2n-2}$ we get
\begin{align*}
2^{2n-2}(5.2^{2n-4}- 6.2^{n-2} +1) & <
 (2^{3n-4}-2^{2n-1} + 3.2^{n-1})^2 - 2.2^{3n-4}(2^{2n-3} - 2^{n-2}-2^{n-1} + 1) \\
& = (2^{3n-4}-2^{2n-1} + 3.2^{n-1})^2 - 2.2^{3n-4}(2^{n-2}-1)(2^{n-1}-1).
\end{align*}
Since $(1-2^n)( 5.2^{2n-4}- 6.2^{n-2} +1) < 0$ and $2.(2^{2n - 1} - 3.2^{n - 1})(2^{n-2}-1)(2^{n-1}-1) > 0$ we have
\begin{align*}
&2^{2n-2}(5.2^{2n-4}- 6.2^{n-2} +1) + (1-2^n)(5.2^{2n-4}- 6.2^{n-2} +1) \\
& ~~~~~~~~~~~~~~~~~~~~~~~~~~ < (2^{3n-4}-2^{2n-1} + 3.2^{n-1})^2 - 2.2^{3n-4}(2^{n-2}-1)(2^{n-1}-1) \\
&~~~~~~~~~~~~~~~~~~~~~~~~~~~~~~~~~~~~~~~~~~~~~~~~~~ + 2.(2^{2n - 1} - 3.2^{n - 1})(2^{n-2}-1)(2^{n-1}-1)\\
& ~~~~~~~~~~~~~~~~~~~~~~~~~~ = (2^{3n-4}-2^{2n-1} + 3.2^{n-1})^2  - 2(2^{3n-4} - 2^{2n - 1} + 3.2^{n - 1} )(2^{n-2}-1)(2^{n-1}-1)\\
& ~~~~~~~~~~~~~~~~~~~~~~~~~~ < (2^{3n-4}-2^{2n-1} + 3.2^{n-1})^2  - 2(2^{3n-4} - 2^{2n - 1} + 3.2^{n - 1} )(2^{n-2}-1)(2^{n-1}-1)\\
&~~~~~~~~~~~~~~~~~~~~~~~~~~~~~~~~~~~~~~~~~~~~~~~~~~ + (2^{n-2}-1)^2(2^{n-1}-1)^2\\
& ~~~~~~~~~~~~~~~~~~~~~~~~~~ = (2^{3n-4}-2^{2n-1} + 3.2^{n-1} - (2^{n-2}-1)(2^{n-1}-1))^2.
\end{align*}
Therefore,
\[
(2^{n-1}-1)^2( 5.2^{n-2}-1)( 2^{n-2} -1) 
 < (2^{3n-4}-2^{2n-1} + 3.2^{n-1} - (2^{n-2}-1)(2^{n-1}-1))^2
\]
noting that 
\[
2^{2n-2}((5.2^{2n-4}- 6.2^{n-2} +1)) + (1-2^n)(5.2^{2n-4}- 6.2^{n-2} +1) = (2^{n-1}-1)^2( 5.2^{n-2}-1)( 2^{n-2} -1).
\]
Hence,
\begin{align*}
( 5.2^{n-2}-1)( 2^{n-2} -1) 
 & < \left(\frac{2^{3n-4}-2^{2n-1} + 3.2^{n-1} - (2^{n-2}-1)(2^{n-1}-1)}{(2^{n-1}-1)}\right)^2\\
 & = \left(\frac{2^{3n-4}-2^{2n-1} + 3.2^{n-1}}{(2^{n-1}-1)}  - (2^{n-2}-1) \right)^2.
\end{align*}
Thus
\[
\sqrt{( 5.2^{n-2}-1)( 2^{n-2} -1)} < \frac{2^{3n-4}-2^{2n-1} + 3.2^{n-1}}{(2^{n-1}-1)}  - (2^{n-2}-1)
\]
and so 
\[
(2^{n-1}-2)+2\sqrt{( 5.2^{n-2}-1)( 2^{n-2} -1)} <  \frac{2( 2^{3n-4}-2^{2n-1} + 3.2^{n-1})  }{(2^{n-1}-1)}.
\]
Hence, the result follows from \cite[Theorem 3.4]{GG} and \cite[Proposition 3.2]{DN}.
\end{proof}
Now, we compare energy and Laplacian energy of the Frobenious group $F_{p, q}$. 
In  \cite[Proposition 3.1]{DN}, $LE(\Gamma_{nc}(F_{p, q}))$ was computed incorrectly. 
The correct expression for $LE(\Gamma_{nc}(F_{p, q}))$ is given below
\begin{equation}\label{corrected-Proposition 3.1-DN}
LE(\Gamma_{nc}(F_{p, q})) =  \frac{2p^2\alpha+2p(q-1)^2}{pq-1},
\end{equation}
where $\alpha = (p-1)(q-1)$.

\begin{theorem}\label{compare-pq}
If $p$ and $q$ are two primes such that $q\mid (p - 1)$ and  $q^2 <  q + p + 1$     then $E(\Gamma_{nc}(F_{p, q})) \leq LE(\Gamma_{nc}(F_{p, q}))$.
\end{theorem}
\begin{proof}
Let $p = qt + 1$ for some integer $t \geq 1$. Then
\begin{align*}
pq - q^2 - p + 1  & = pq - (p - 1) - q^2\\
& = q(p - t - q) = q(q - 1)(t - 1) \geq 0
\end{align*}
and 
\begin{align*}
p^4q -p^3q^2 -p^4+p^3+2p^2-p^2q & = p^2(p^2(q - 1) - p(q^2 - 1) - q + 2) \\
& = p^2(q - 1)\left(p^2 - p(q + 1) - 1 + \frac{1}{q - 1}\right)\\
& = p^2(q - 1)\left(p(p - 1) -pq - 1 +  \frac{1}{q - 1}\right)\\
& = p^2(q - 1)\left(pq(t - 1)   - 1 +  \frac{1}{q - 1}\right) \geq 0
\end{align*}
noting that $t = 1$ if and only if $p = 3$ and $q = 2$. Since
$q^2 < q + p + 1$ we have
%
\[
pq - 2p - q <  pq - q^2 - p + 1  \leq p^2(pq - q^2 - p + 1)
\]
and so
\begin{align*}
& -q < p^3q - p^2q^2 -p^3 + p^2 + 2p - pq\\
\Rightarrow & -pq < p^4q - p^3q^2 - p^4 + p^3 + 2p^2 - p^2q\\
\Rightarrow & -pq +1 < p^4q - p^3q^2 - p^4 + p^3 + 2p^2 - p^2q.
\end{align*}
%
%
%
%
Adding $p^2q^2 - pq$ on both sides we get
\begin{align*}
(pq - 1)^2 & <  (pq - p - q + 1)(p^3 - p^2q) + 2p^2 - pq\\
& = (p - 1)(q - 1)(p^3 - p^2q) + 2p^2 - pq = \alpha(p^3 - p^2q) + 2p^2 - pq,
\end{align*} 
where $\alpha = (p - 1)(q - 1)$. 
%
Since $\alpha(p^3 - p^2q) + 2p^2 - pq > 0$ and $(q^2 - 2q)(2p^2 - pq) + p\alpha  + (q - 1)^2 > 0$ we have
\begin{align*}
(pq - 1)^2  & <  \alpha(p^3 - p^2q) + 2p^2 - pq  + (q^2 - 2q)(2p^2 - pq) + p \alpha + (q - 1)^2\\
& = \alpha(p^3 - p^2q + p) + (q^2 - 2q + 1)(2p^2 - pq) + (q - 1)^2\\
& = \alpha (p^3 -p(pq-1))+(q-1)^2(2p^2-pq+1).
\end{align*}    
%
%
Therefore,    
\begin{align*}
\alpha(pq - 1)^2   & < \alpha^2(p^3 - p(pq - 1)) + \alpha(q - 1)^2(2p^2 - (pq - 1)) \\
& = p^3 \alpha^2 + 2p^2 \alpha (q - 1)^2  - p(pq - 1)\alpha^2- \alpha(pq - 1) (q - 1)^2\\
& = p^3 \alpha^2 + 2p^2 \alpha (q-1)^2 + p(q-1)^4 -p (pq-1)\alpha^2- \alpha(pq-1) (q-1)^2,
\end{align*} 
since $p(q - 1)^4 > 0$.  Now, multiplying both sides by $\frac{4p}{(pq - 1)^2}$, we get 
%
%
\begin{align*}   
& 4p\alpha  <  \frac{4p^4\alpha^2 + 8p^3 \alpha (q - 1)^2 + 4p^2(q - 1)^4 - 4p^2 (pq - 1)\alpha^2 - 4p \alpha(pq - 1)(q - 1)^2}{(pq - 1)^2} \\
\Rightarrow & 4p\alpha + \alpha^2 <  \frac{4p^4 \alpha^2 + 8p^3 \alpha (q - 1)^2 + 4p^2(q - 1)^4 - 4p^2 (pq - 1)\alpha^2 - 4p \alpha(pq - 1)(q - 1)^2}{(pq - 1)^2} + \alpha^2\\
& ~~~~~~~~~~~~  = \left(\frac{2p^2\alpha + 2p(q - 1)^2}{pq - 1} - \alpha\right)^2.
\end{align*}     
Therefore,
\[
\sqrt{4p\alpha +\alpha^2} <  \frac{2p^2 \alpha+ 2p(q-1)^2}{pq-1} - \alpha.
\]   
Hence, the result follows from   \eqref{thm11} and   \eqref{corrected-Proposition 3.1-DN}.
\end{proof}
\begin{remark}
The proof of Theorem \ref{compare-pq} also shows that  $E(\Gamma_{nc}(F_{p, q})) > LE(\Gamma_{nc}(F_{p, q}))$ if  $q\mid (p - 1)$ and  $q^2 >  q + p + 1$. Thus, \eqref{conjecture 1} does not hold for $\Gamma_{nc}(F_{p, q})$ if $q^2 >  q + p + 1$. For example, if $p = 43, q = 7$; $p = 53, q = 13$; $p = 67, q = 11$; $p = 89, q = 11$ etc. then the non-commuting graph of the group $F_{p, q}$ does not satisfy \eqref{conjecture 1}. This gives a new family of counter examples for the conjecture given in  \eqref{conjecture 1}, viz. $K_{1.(p - 1), p.(q - 1)}$ where $p$ and $q$ are two primes such that $q\mid (p - 1)$ and  $q^2 >  q + p + 1$. 
\end{remark}

\begin{theorem}\label{compare-2}
Let $G$ be a finite group with center $Z(G)$. If $\frac{G}{Z(G)} \cong {\mathbb{Z}}_p \times {\mathbb{Z}}_p$ 
then $E(\Gamma_{nc}(G)) \leq LE(\Gamma_{nc}(G))$. 
\end{theorem}
\begin{proof}
If $\frac{G}{Z(G)} \cong {\mathbb{Z}}_p \times {\mathbb{Z}}_p$  then $LE(\Gamma_{nc}(G)) = 2p(p - 1)|Z(G)|$,   by \cite[Theorem 2.2]{DN}.  Therefore, by Theorem \ref{thm3}(c), we have $E(\Gamma_{nc}(G)) \leq LE(\Gamma_{nc}(G))$.    
%
\end{proof}

In \cite{DN}, $LE(\Gamma_{nc}(G))$ was computed if $\frac{G}{Z(G)} \cong D_{2m}$. However, the expression for $LE(\Gamma_{nc}(G))$ given in \cite[Theorem 2.4]{DN} is not correct and so all its corollaries. The correct expression for $LE(\Gamma_{nc}(G))$  is 
\begin{equation}\label{Theorem 2.4-DN-CV}
LE(\Gamma_{nc}(G)) =  \frac{2mn^2(m-1)(m-2) +2mn(2m - 1)}{2m-1},
\end{equation}
where $n = |Z(G)|$. Further, the corrected versions of \cite[Corollaries 2.5--2.7]{DN} are given by
%
\begin{equation}
LE(\Gamma_{nc}(M_{2rs}) = \begin{cases}
\frac{2r(r-1)(r-2)s^2 +2rs(2r - 1)}{2r-1}, &  \text{ if $r$ is odd}\\
\frac{rs^2(r-2)(r-4) +2rs(r - 1)}{r-1}, &  \text{ if $r$ is even},
\end{cases}
\end{equation}
%
%
\begin{equation}
LE(\Gamma_{nc}(D_{2m}) =  \begin{cases} 
\frac{2m(m-1)(m-2) +2m(2m - 1)}{2m-1}, & \text{ if $m$ is odd}\\
\frac{m(m-2)(m-4) +2m(m - 1)}{m-1}, &\text{ if $m$ is even}
\end{cases}
\end{equation}
%
and 
\begin{equation}
LE(\Gamma_{nc}(Q_{4m}) =  \frac{8m(m-1)(m-2) +4m(2m - 1)}{2m-1}.
\end{equation}

\begin{theorem}\label{compare-2D2m}
Let $G$ be a finite group with center $Z(G)$. If $\frac{G}{Z(G)} \cong D_{2m}$ then $E(\Gamma_{nc}(G)) \leq LE(\Gamma_{nc}(G))$.
\end{theorem}
\begin{proof}
%
%
%
%
%
Since $m \geq 2$ we have $(2m-1)^2 \leq (2m-1)(m^2 - 1)$. Therefore,
\[
(2m-1)^2 \leq mn^2(m-1)^2(m-2) + n(2m-1)(m^2-1),
\]
where $n = |Z(G)|$.
Multiplying both sides by $4m(m-2)$ we get
\begin{align*}
& 4m(m-2)(2m-1)^2\leq 4m^2n^2(m-1)^2(m-2)^2 + 4mn(2m-1)(m-2)(m^2-1)\\
\Rightarrow & (2m-1)^2(m-1)(5m-1)  - (2m-1)^2(m+1)^2 \\
&~~~~~~~~~~~~~~~~~~~~~~~~~~~~~~~~~~\leq 4m^2n^2(m-1)^2(m-2)^2 + 4mn(2m-1)(m-2)(m^2-1) \\
\Rightarrow & (2m-1)^2(m-1)(5m-1) \leq  (2mn(m-1)(m-2)+ (2m-1)(m+1))^2.
\end{align*}
Therefore,
\begin{align*}
(2m-1)\sqrt{(m-1)(5m-1)}  & \leq 2mn(m-1)(m-2)+ (2m-1)(m+1)\\
& = 2mn(m-1)(m-2)+ 2m(2m-1)-(2m-1)(m-1).
\end{align*}
Hence,
\[
(2m-1)\left(\sqrt{(m-1)(5m-1)} + (m-1)\right)  \leq 2mn(m-1)(m-2)+ 2m(2m-1)
\]
and so the result follows from Theorem \ref{thm10} and \eqref{Theorem 2.4-DN-CV}.
%
%
%
%
%
    
\end{proof}

\begin{corollary}\label{compare-cor-1}
If $G = M_{2rs}, D_{2m}, Q_{4m}$ or $U_{6n}$  then $E(\Gamma_{nc}(G)) \leq LE(\Gamma_{nc}(G))$.
\end{corollary}

\begin{proof}
If $G = M_{2rs}, D_{2m}, Q_{4m}$ or $U_{6n}$  then $\frac{G}{Z(G)}$ is isomorphic to a dihedral group. Hence, the result follows from Theorem \ref{compare-2D2m}.
\end{proof}

\begin{corollary}\label{order-16}
Consider the following groups of order $16$:
${\mathbb{Z}}_2 \times D_8$,
${\mathbb{Z}}_2 \times Q_8$,
${\mathcal{M}}_{16}  = \langle x , y \mid  x^8 = y^2 = 1, yxy = x^5 \rangle$,
${\mathbb{Z}}_4 \rtimes {\mathbb{Z}}_4 = \langle x, y \mid  x^4 = y^4 = 1, yxy^{-1} = x^{-1} \rangle$,
$D_8 * {\mathbb{Z}}_4 = \langle x, y, c \mid  x^4 = y^2 = c^2 =  1, xy = yx, xc = cx, yc = x^2cy \rangle$ and
$SG(16, 3)  = \langle x, y \mid  x^4 = y^4 = 1, xy = y^{-1}x^{-1}, xy^{-1} = yx^{-1}\rangle$.
If $G$ is any of the above listed group then $E(\Gamma_{nc}(G)) \leq LE(\Gamma_{nc}(G))$.
\end{corollary}
\begin{proof}
For all the  groups listed above $\frac{G}{Z(G)}$ is isomorphic to ${\mathbb{Z}}_2 \times {\mathbb{Z}}_2$. Hence, the result follows from Theorem \ref{compare-2}.
\end{proof}

\begin{corollary}
Let $p$ be any prime and let $G$ be a finite non-abelian group. If $G$ has one of the following properties then $E(\Gamma_{nc}(G)) \leq LE(\Gamma_{nc}(G))$.
\begin{enumerate}
\item $|G| = p^3$.
\item $|\cent(G)| = 4$.
\item $|\cent(G)| = 5$.
\item $|G| = p^n$ and $|\cent(G)| = (p + 2)$.
\end{enumerate}
Here,  $\cent(G) = \{C_G(x) : x \in G\}$.
\end{corollary}
\begin{proof}
If $G$ has one of the  properties listed above then
 $\frac{G}{Z(G)}$ is isomorphic to ${\mathbb{Z}}_p \times {\mathbb{Z}}_p$ or $D_{2p}$, for some prime $p$ (see \cite{BS}). Hence, the result follows from Theorem \ref{compare-2} and Theorem \ref{compare-2D2m}.
\end{proof}

\begin{corollary}
If $G$ has one of the following properties then $E(\Gamma_{nc}(G)) \leq LE(\Gamma_{nc}(G))$.
\begin{enumerate}
\item $\Pr(G) \in \{\frac{5}{14}, \frac{2}{5}, \frac{11}{27}, \frac{1}{2}, \frac{5}{8}\}$.
\item $\Pr(G) = \frac{p^2 + p - 1}{p^3}$ and $p$ is the smallest prime dividing $|G|$.
\end{enumerate}
Here, $\Pr(G)$ is the commuting probability of $G$.
\end{corollary}
\begin{proof}
If $G$ has one of the  properties listed above then
 $\frac{G}{Z(G)}$ is isomorphic to ${\mathbb{Z}}_2 \times {\mathbb{Z}}_2$ or $D_{2m}$, for some $m$. Hence, the result follows from Theorem \ref{compare-2} and Theorem \ref{compare-2D2m}.
\end{proof}

\begin{proposition}
Let $G$ be a finite non-abelian group such that $\Gamma_{nc}(G)$ is   planar. Then $E(\Gamma_{nc}(G)) \leq LE(\Gamma_{nc}(G))$. 
\end{proposition}
\begin{proof}
If $\Gamma_{nc}(G)$ is   planar then, by \cite[Theorem 3.1]{AF14}, $\frac{G}{Z(G)}$ is isomorphic to $D_6, D_8$ or $Q_8$. Hence, the result follows from  Corollary \ref{compare-cor-1}.
\end{proof}


\begin{proposition}
Let $G$ be a finite non-abelian group such that the complement of $\Gamma_{nc}(G)$ is   planar. Then $E(\Gamma_{nc}(G)) \leq LE(\Gamma_{nc}(G))$.
\end{proposition}
\begin{proof}
If the complement of $\Gamma_{nc}(G)$ is   planar then, by  \cite[Theorem 2.2]{AF14} or \cite[Theorem 5.7]{das2},  we have  $G$ is isomorphic to either $D_{2n}$ for $n = 3, 4, 5, 6$;   $Q_{4m}$ for $m = 2, 3$; $A_n$ for $n = 4, 5$; $S_z(2) = \langle x, y : x^5 = y^4 =1, y^{-1}xy=x^2\rangle$, $SL(2, 3), S_4$ or any group listed in Corollary \ref{order-16}.

If $G$ is isomorphic to either $D_{2n}$ for $n = 3, 4, 5, 6$;   $Q_{4m}$ for $m = 2, 3$; or any group listed in Corollary \ref{order-16} then by Theorem \ref{compare-2},  Corollary \ref{compare-cor-1} and Corollary \ref{order-16} we have 
\[
E(\Gamma_{nc}(G)) \leq LE(\Gamma_{nc}(G)).
\] 

If $G \cong A_4$ then, as shown in the proof of   \cite[Theorem 3.2]{DN1},  $\Gamma_{nc}(G)$ is the complement of $K_3 \sqcup 4K_2$. The characteristic polynomial of  $\mathcal{A}(\Gamma_{nc}(G))$ is given by $x^6(x + 2)^3(x^2 -6x -24) = 0$. Therefore, 
\[
\spec(\Gamma_{nc}(G)) = \{0^6, (-2)^3, (3 \pm \sqrt{33})^1\}
\]
 and so $E(\Gamma_{nc}(G)) =   6 + 2\sqrt{33}$.           
It is also observed that $A_4$ is an AC-group with six distinct centralizers of non-central elements, one having order $4$ and the rest having order $3$. Let $X_1, X_2, \dots, X_5$ be the distinct centralizers of non-central element of $A_4$ such that  $|X_1|= |X_2|=|X_3|=|X_4|= 3$ and $|X_5|=4$. Then, by  \cite[Theorem 2.4]{DN-2018}, we have 
\[
\lspec(\Gamma_{nc}(G)) = \{0^1, 8^2, 9^{4}, 11^{4}\}.
\]
Since, $|v(\Gamma_{nc}(G))|= 11$ and $|e(\Gamma_{nc}(G))|= 48$ we have $\frac{2|e(\Gamma_{nc}(G))|}{|v(\Gamma_{nc}(G))|} = \frac{96}{11}$. We have
             $|0-\frac{96}{11}|= \frac{96}{11}$, $|8-\frac{96}{11}|= \frac{8}{11}$, $|9-\frac{96}{11}|= \frac{3}{11}$ and $|11-\frac{96}{11}|= \frac{25}{11}$. Therefore, 
\[            
LE(\Gamma_{nc}(G)) = \frac{96}{11}+ 2\times\frac{8}{11}+ 4\times\frac{3}{22}+ 4\times\frac{25}{11}= \frac{224}{11}.
\]                 
Hence, $E(\Gamma_{nc}(G)) < LE(\Gamma_{nc}(G))$.

If $G \cong A_5$ then  by  \cite[Theorem 2.2]{GG}, noting that $A_5 = PSL(2, 4)$,  we  get
\[
\spec(\Gamma_{nc}(G))=\{ 0^{38}, (-4)^{5}, 2^9, 3^4, (x_1)^1, (x_2)^1, (x_3)^1\},
\]
where $x_1, x_2$ and $x_3$ are the roots of $x^3 -50x^2 - 324x -480 = 0$. Therefore,
\[
E(\Gamma_{nc}(G)) =  5 \times 4 + 9 \times 2 + 4 \times 3 +  |x_1| + |x_2|+ |x_3| = 50+  |x_1| + |x_2|+ |x_3|.
\]
Since  $x_1 \approx -2.467, x_2 \approx -3.478$ and $x_3 \approx 55.945$ we have 
\[
E(\Gamma_{nc}(G)) \approx 50 + 2.467 + 3.478 + 55.945 = 111.89.
\] 
Again, by  \cite[Proposition 4.3]{DN-2018} and  \cite[Proposition 3.3]{DN}, we get
\[
\lspec(\Gamma_{nc}(G))=\{0^1, (55)^{18}, (56)^{10}, (57)^{10}, (59)^{20}\}
\]
and $LE(\Gamma_{nc}(G)) = \frac{8580}{59}  \approx 145.423$. Therefore,  
$E(\Gamma_{nc}(G)) <  LE(\Gamma_{nc}(G))$.

If $G \cong Sz(2)$ then $Z(G) = \{1\}$. Therefore, as shown in the proof of   \cite[Theorem 2.2]{DN1},  $\Gamma_{nc}(G)$ is the complement of $K_4 \sqcup 5K_3$. The characteristic polynomial of  $\mathcal{A}(\Gamma_{nc}(G))$ is given by $x^{13}(x + 3)^4(x^2 - 12x - 60) = 0$. Therefore, 
\[
\spec(\Gamma_{nc}(G)) = \{0^{13}, (-3)^4, (6 \pm 4\sqrt{6})^1\}
\]  
and so $E(\Gamma_{nc}(G)) = 12 + 6 + 4\sqrt{6} - 6 + 4\sqrt{6} = 12 + 8\sqrt{6}$.

The distinct centralizers of non-central elements of $Sz(2)$ are given by 
\begin{align*} 
X_1 & = C_G(xy) = \{1, xy, x^4b^2, x^3y^3\}\\
X_2 & = C_G(x^2y) = \{1, x^2y, x^3y^2, xy^3\}\\
X_3 & = C_G(x^2y^3) = \{1, x^2y^3, xy^2, x^4y\}\\
X_4 & = C_G(y) = \{1, y, y^2, y^3\}\\
X_5 & = C_G(x^3y) = \{1, x^3y, x^2y^2, x^4y^3\}\\
X_6 & = C_G(x) = \{1, x, x^2, x^3, x^4\}.
\end{align*}
Therefore, $|X_1|= \cdots = |X_5|= 4$  and $|X_6|= 5$. By  \cite[Theorem 2.4]{DN-2018}, we have 
\[
\lspec(\Gamma_{nc}(G)) = \{0^1, (15)^{3}, (16)^{10}, (19)^5\}.
\]  
We have $|v(\Gamma_{nc}(G))|= 19$ and $|e(\Gamma_{nc}(G))|= 150$ we have $\frac{2|e(\Gamma_{nc}(G))|}{|v(\Gamma_{nc}(G))|} = \frac{300}{19}$. We have
             $|0-\frac{300}{19}|= \frac{300}{19}$, $|15 - \frac{300}{19}|=  \frac{15}{19}$, $|16-\frac{300}{19}|= \frac{4}{19}$ and $|19-\frac{300}{19}|= \frac{61}{19}$. Therefore,
\[            
LE(\Gamma_{nc}(G)) = \frac{300}{19}   +  3 \times \frac{15}{19} + 10 \times \frac{4}{19} + 5 \times  \frac{61}{19} =  \frac{690}{19}.
\]             
Hence, $E(\Gamma_{nc}(G)) < LE(\Gamma_{nc}(G))$.              

If $G \cong SL(2, 3)$ then as shown in the proof of   \cite[Theorem 3.2]{DN1},  $\Gamma_{nc}(G)$ is the complement of $3K_2 \sqcup 4K_4$. The characteristic polynomial of  $\mathcal{A}(\Gamma_{nc}(G))$ is given by $x^{15}(x +2)^2(x + 4)^4(x^2 - 16x -48) = 0$. Therefore, 
\[
\spec(\Gamma_{nc}(G)) = \{0^{15}, (-2)^2, (-4)^3, (8 \pm 4\sqrt{7})^1\}
\]  
and so $E(\Gamma_{nc}(G)) = 16 + 8\sqrt{7}$. It is also observed that $SL(2, 3)$ is an AC-group with seven distinct centralizers of non-central elements, three having order $4$ and the rest having order $6$. Let $X_1, X_2, \dots, X_7$ be the distinct centralizers of non-central element of $SL(2, 3)$ such that  $|X_1|= |X_2|=|X_3|= 4$ and $|X_4|= \cdots = |X_7|=6$. Then, by  \cite[Theorem 2.4]{DN-2018}, we have 
\[
\lspec(\Gamma_{nc}(G)) = \{0^1, (18)^{12}, (20)^3, (22)^6\}.
\]  
Since, $|v(\Gamma_{nc}(G))|= 22$ and $|e(\Gamma_{nc}(G))|= 204$ we have $\frac{2|e(\Gamma_{nc}(G))|}{|v(\Gamma_{nc}(G))|} = \frac{204}{11}$. We have
             $|0-\frac{204}{11}|= \frac{204}{11}$, $|18 - \frac{204}{11}|= \frac{6}{11}$, $|20-\frac{204}{11}|= \frac{16}{11}$ and $|22-\frac{204}{11}|= \frac{38}{11}$. Therefore,
\[            
LE(\Gamma_{nc}(G)) = \frac{204}{11}   +  12 \times \frac{6}{11} + 3 \times \frac{16}{11} + 6 \times  \frac{38}{11} =  \frac{552}{11}.
\]             
Hence, $E(\Gamma_{nc}(G)) < LE(\Gamma_{nc}(G))$.             

If $G \cong S_4$ then, using GAP \cite{GAP2017}, we find  the characteristic polynomial of  $\mathcal{A}(\Gamma_{nc}(G))$ as given below: 
\[
x^{10}(x+2)^6(x^2+2x-4)^2(x^3-16x^2-76x-48) = 0 
\]
and so 
\[
\spec(\Gamma_{nc}(G)) = \{0^{10}, (-2)^6, (-1 + \sqrt{5})^2, (-1 - \sqrt{5})^2, (\alpha_1)^1, (\alpha_2)^1, (\alpha_3)^1\},
\]
where $\alpha_1, \alpha_2, \alpha_3$ are the roots of $x^3-16x^2-76x-48 = 0$. Therefore, 
\begin{align*}
E(\Gamma_{nc}(G)) & = 6 \times 2 + 2 \times (-1 + \sqrt{5})  +  2 \times (1 + \sqrt{5}) + |\alpha_1| + |\alpha_2| + |\alpha_3|\\
& = 12 + 4\sqrt{5} + |\alpha_1| + |\alpha_2| + |\alpha_3|.
\end{align*} 
Since $\alpha_1 \approx -0.758, \alpha_2 \approx -3.175$ and $\alpha_3 \approx 19.933$ we have $E(\Gamma_{nc}(G)) \approx 35.866 + 4\sqrt{5}$.

We have $\lspec(\Gamma_{nc}(G)^c) = \{0^5, (4 - \sqrt{13})^2, 1^3, 2^4, 3^6, 5^1,(4 + \sqrt{13})^2\}$, where $\Gamma_{nc}(G)^c$ is the commuting graph of $G$ (see \cite[Page 85]{DN3}). Therefore, by \cite[Theorem 2.2]{DN-2018}, we have
\[
\lspec(\Gamma_{nc}(G)) = \{0^1, (19 - \sqrt{13})^2, (18)^1, (20)^6, (21)^4, (22)^3, (19 + \sqrt{13})^2, (23)^{4}\}.
\]
We have 
$|v(\Gamma_{nc}(G))|= 23$, $|e(\Gamma_{nc}(G))|= 228$ and   $\frac{2|e(\Gamma_{nc}(G))|}{|v(\Gamma_{nc}(G))|} = \frac{456}{23}$. Therefore,
$|0- \frac{456}{23}| = \frac{456}{23}$, 
$|19 - \sqrt{13} - \frac{456}{23}| = |- \frac{19}{23}- \sqrt{13}| = \frac{19}{23} + \sqrt{13}$, 
$|18 - \frac{456}{23}| = |-\frac{42}{23}| = \frac{42}{23}$, 
$|20- \frac{456}{23}|  = \frac{4}{23}$, 
$|21- \frac{456}{23}| =  \frac{27}{23}$, 
$|22- \frac{456}{23}| = \frac{50}{23}$,
$|19 + \sqrt{13} - \frac{456}{23}| = |- \frac{19}{23}+ \sqrt{13}| = - \frac{19}{23}+ \sqrt{13}$ and $|23 - \frac{456}{23}| = \frac{73}{23}$. Hence,  
\begin{align*}
LE(\Gamma_{nc}(G)) & =  \frac{456}{23} + 2 \times (\frac{19}{23} + \sqrt{13}) + \frac{42}{23} +  6 \times \frac{4}{23} +  4 \times \frac{27}{23} + 3 \times \frac{50}{23} + 2 \times (- \frac{19}{23}+ \sqrt{13}) + 4 \times \frac{73}{23}\\
& = \frac{456 + 42 + 24 + 108 + 150 + 292}{23}  + 4\sqrt{13} = \frac{1072}{23}  + 4\sqrt{13}.
\end{align*}
Thus, $E(\Gamma_{nc}(G)) < LE(\Gamma_{nc}(G))$.
This completes the proof.
\end{proof}

\begin{proposition}
Let $G$ be a finite non-abelian group such that the complement of $\Gamma_{nc}(G)$ is   toroidal or projective. Then $E(\Gamma_{nc}(G)) \leq LE(\Gamma_{nc}(G))$.
\end{proposition}
\begin{proof}
If the complement of $\Gamma_{nc}(G)$ is   toroidal or projective then by \cite[Theorem 3.3]{DN1} we have $G$ is isomorphic to either $D_{2n}$ for $n = 7, 8$; $Q_{16}, QD_{16}$, $F_{7, 3}$, $D_{6} \times {\mathbb{Z}}_3$ or $A_{4} \times {\mathbb{Z}}_2$.

If $G$ is isomorphic to either $D_{2n}$ for $n = 7, 8$; $Q_{16}$,  $QD_{16}$ or $F_{7, 3}$  then by Theorem \ref{compare-2},  Corollary \ref{compare-cor-1}, Theorem \ref{compare-QD} and Theorem \ref{compare-pq}  we have 
\[
E(\Gamma_{nc}(G)) \leq LE(\Gamma_{nc}(G)).
\]

Suppose that  $G \cong D_{6} \times {\mathbb{Z}}_3$.  It is easy to see that $D_{6} \times {\mathbb{Z}}_3$ is an AC-group with four centralizers of non-central elements namely, $X_1 = \{1, b\} \times {\mathbb{Z}}_3$, $X_2 = \{1, ab\} \times {\mathbb{Z}}_3$, $X_3 = \{1, a^2b\} \times {\mathbb{Z}}_3$ and $X_4 = \{1, a, a^2\} \times {\mathbb{Z}}_3$. Therefore, $|X_1| = |X_2| = |X_3| = 6$ and $|X_4| = 9$ and so, by  \cite[Lemma 2.1]{DN1}, $\Gamma_{nc}(G)$ is the complement of $3K_3\sqcup K_6$. Therefore,            
\[
\mathcal{A}(\Gamma_{nc}(G)) = \begin{bmatrix}
      0_{2\times 2}                  & J_{2 \times 3}\\
      J_{3 \times 2}       & (J - I)_{3 \times 3}
      \end{bmatrix}  \otimes   J_{3 \times 3} = B_{5 \times 5 } \otimes J_{3 \times 3}.
\]
By Lemma \ref{lem7} we have $\spec(J_{3 \times 3})= \{0^2, 3^1\}$ and $\spec(B_{5 \times 5})= \{0^1, (-1)^2 , (1\pm \sqrt{7})^1\}$. Therefore, by Theorem \ref{thm1}, we have $\spec(\Gamma_{nc}(G))=\{0^{11}, (-3)^2, (3 \pm 3\sqrt{7})^1\}$ and so $E(\Gamma_{nc}(G)) = 6 + 6\sqrt{7}$.

By  \cite[Theorem 2.4]{DN-2018}, we have  $\lspec(\Gamma_{nc}(G))=\{0^1, 9^5, 12^6, 15^3\}$. We have $|v(\Gamma_{nc}(G))|= 15$, $|e(\Gamma_{nc}(G))|= 81$  and so $\frac{2|e(\Gamma_{nc}(G))|}{|v(\Gamma_{nc}(G))|}= \frac{162}{15}$. Therefore,
             $|0-\frac{162}{15}|= \frac{162}{15}$, $|9-\frac{162}{15}|= \frac{27}{15}$, $|12-\frac{162}{15}|= \frac{18}{15}$ and $|15-\frac{162}{15}|= \frac{63}{15}$. Hence,
\[            
LE(\Gamma_{nc}(G))= \frac{162}{15}+ 5\times\frac{27}{15}+ 6\times\frac{18}{15}+ 3 \times \frac{63}{15}= \frac{594}{15}. 
\]
Thus, $E(\Gamma_{nc}(G)) < LE(\Gamma_{nc}(G))$.

Suppose that $G \cong A_{4} \times {\mathbb{Z}}_2$.  It can be seen that $G \cong A_{4} \times {\mathbb{Z}}_2$ is an AC-group with five centralizers of non-central elements namely, $X_1 = \{1, y, y^2\} \times {\mathbb{Z}}_2$, $X_2 = \{1, xy, y^2x\} \times {\mathbb{Z}}_2$, $X_3 = \{1, yx, xy^2\} \times {\mathbb{Z}}_2$,  $X_4 = \{1, xyx, yxy\} \times {\mathbb{Z}}_2$ and $X_5 = \{1, x,  yxy^2, y^2xy\} \times {\mathbb{Z}}_2$ since $A_4 = \langle x, y : x^2 = y^3 = (xy)^3 = 1\rangle$. Therefore, $|X_1| = |X_2| = |X_3| = |X_4|= 6$ and $|X_5| = 8$ and so, by  \cite[Lemma 2.1]{DN1}, $\Gamma_{nc}(G)$ is the complement of $4K_4\sqcup K_6$. Therefore, the characteristic polynomial of  $\mathcal{A}(\Gamma_{nc}(G))$ is given by 
\[
x^{17}(x+4)^3(x^2 -12x-96) = 0 
\]
and so 
\[
\spec(\Gamma_{nc}(G)) = \{0^{17}, (-4)^3, (6 \pm 2\sqrt{33})^1\}.
\]
Hence, $E(\Gamma_{nc}(G)) = 12 +  6 + 2\sqrt{33} -6 + 2\sqrt{33} = 12 + 4\sqrt{33}$. 

By  \cite[Theorem 2.4]{DN-2018}, we have  $\lspec(\Gamma_{nc}(G))=\{0^1, 16^{5}, 18^{12},  22^4\}$. We have $|v(\Gamma_{nc}(G))|= 22$, $|e(\Gamma_{nc}(G))|= 192$  and so $\frac{2|e(\Gamma_{nc}(G))|}{|v(\Gamma_{nc}(G))|}=  \frac{192}{11}$. Therefore,
$|0 - \frac{192}{11}| = \frac{192}{11}$, $|16 - \frac{192}{11}| = \frac{16}{11}$, $|18 - \frac{192}{11}| = \frac{6}{11}$ and $|22 - \frac{192}{11}| = \frac{50}{11}$. Hence,
\[
LE(\Gamma_{nc}(G)) = \frac{192}{11} + 5 \times \frac{16}{11} + 12 \times \frac{6}{11} + 4 \times \frac{50}{11} = \frac{544}{11}.
\]
Thus, $E(\Gamma_{nc}(G)) < LE(\Gamma_{nc}(G))$. This completes the proof.
\end{proof}
We conclude this paper with the following result. 
 \begin{theorem}\label{thm17}
If $G = A(n,v)$ or $A(n,p)$  then 
      $E(\Gamma_{nc}(G)) = LE(\Gamma_{nc}(G))$.
\end{theorem} 
\begin{proof}
If $G = A(n,v)$ then by \cite[Proposition 3.5]{DN} and Theorem \ref{thm12}, we have  
\[
LE(\Gamma_{nc}(G)) = 2^{2n+1} - 2^{n+2} = 2^n(2^{n + 1} -4) = E(\Gamma_{nc}(G)).
\]
If $G = A(n,p)$ then  by \cite[Proposition 3.6]{DN} and Theorem \ref{thm13}, we have  
\[
LE(\Gamma_{nc}(G)) = 2(p^{3n} - p^{2n}) = 2p^{2n}(p^{n} - 1) = E(\Gamma_{nc}(G)).
\]
Hence, the result follows.
\end{proof}

\noindent {\bf Acknowledgements}
Discussions with Prof. William C. Jagy, Prof. Dmitry Ezhov, Prof. J. W. Tanner and others through https://math.stackexchange.com were helpful  in developing  Section 3 of this paper. The authors are thankful to them.    The first author is thankful to Indian Council for Cultural Relations for the ICCR Scholarship.  

%
%
%
%
%
%
%
%
%
%

\end{document}